\documentclass[11pt,reqno]{amsart}
\usepackage[latin1]{inputenc}
\usepackage{amsmath}
\usepackage{amsfonts}
\usepackage{amssymb,amsbsy,amsthm}
\usepackage{hyperref}
\usepackage{geometry}
\usepackage{color}
\usepackage{ytableau,shuffle}
\usepackage[T1]{fontenc}
\usepackage[enableskew,vcentermath]{youngtab}
\usepackage{mathtools}
\usepackage{tikz}
\usetikzlibrary{positioning}

\newcommand{\red}[1]{\textcolor{red}{#1}}

\newcommand{\Q}{\mathcal Q}

\newcommand{\TTT}{\mathcal T}
\newcommand{\symm}{S}

\DeclareMathOperator\Des{Des}

\DeclareMathOperator\cDes{cDes}
\DeclareMathOperator\CDes{CDES}
\DeclareMathOperator\MDes{GDes}
\DeclareMathOperator\cMDes{cGDes}
\DeclareMathOperator\crn{cr}
\DeclareMathOperator\nest{ne}
\DeclareMathOperator\um{um}

\DeclareMathOperator\height{ht}
\DeclareMathOperator\oddcol{oc}

\DeclareMathOperator\Fix{Fix}

\def\NN{{\mathbb N}}

\def\I{{\mathcal I}}
\def\M{{\mathcal M}}

\def\SYT{{\bf{\rm SYT}}}
\def\O{{\bf{\rm O}}}

\theoremstyle{plain}
\newtheorem{theorem}{Theorem}[section]
\newtheorem{proposition}[theorem]{Proposition}
\newtheorem{lemma}[theorem]{Lemma}
\newtheorem{corollary}[theorem]{Corollary}
\newtheorem{conjecture}[theorem]{Conjecture}
\newtheorem{problem}[theorem]{Problem}
\theoremstyle{definition}
\newtheorem{definition}[theorem]{Definition}

\newtheorem{example}[theorem]{Example}

\newtheorem{remark}[theorem]{Remark}
\newtheorem{observation}[theorem]{Observation}

\title{Cyclic descents, matchings and Schur-positivity}

\author{Ron M. Adin}
\address{Department of Mathematics, Bar-Ilan University, 
Ramat-Gan 52900, Israel}
\email{radin@math.biu.ac.il}

\author{Yuval Roichman}
\address{Department of Mathematics, Bar-Ilan University, 
Ramat-Gan 52900, Israel}
\email{yuvalr@math.biu.ac.il}

\date{January 1, 2023}

\thanks{Both authors partially supported by the Israel Science Foundation, Grant No.\ 1970/18.}

\begin{document}

\maketitle

\begin{abstract}
	A new descent set statistic on involutions, 
	defined geometrically via their interpretation as matchings, 
	is introduced in this paper, and shown to be 
	equidistributed with the standard one. This concept is then applied to  
	construct explicit cyclic descent extensions on involutions,
	standard Young tableaux 
	and Motzkin paths.  
	Schur-positivity of the associated quasisymmetric functions follows. %

\end{abstract}

\tableofcontents

\section{Introduction}\label{sec:introduction}

The notion of descent set, for permutations as well as for standard Young tableaux, is well established. 
Klyachko~\cite{Klyachko} and Cellini~\cite{Cellini} introduced a natural notion of cyclic descents for permutations.
This notion was generalized to standard Young tableaux of rectangular shapes by Rhoades~\cite{Rhoades}, and to other shapes and 
combinatorial sets 
in~\cite{ARR}.

\smallskip

For a positive integer $n$, denote $[n] := \{1, \ldots, n\}$.

\begin{definition}\label{def:cDes}
	Let $\TTT$ be a finite set, equipped with any set map 
	$\Des: \TTT \longrightarrow 2^{[n-1]}$. 
	A {\em cyclic extension} of $\Des$ is
	a pair $(\cDes,p)$, where 
	$\cDes: \TTT \longrightarrow 2^{[n]}$ is a map 
	and $p: \TTT \longrightarrow \TTT$ is a bijection,
	satisfying the following axioms:  for all $T$ in  $\TTT$,
	\[
	\begin{array}{rl}
		\text{(extension)}   & \cDes(T) \cap [n-1] = \Des(T),\\
		\text{(equivariance)}& \cDes(p(T))  = 1+\cDes(T) \pmod n,\\
		\text{(non-Escher)}  & \varnothing \subsetneq \cDes(T) \subsetneq [n],\\
	\end{array}
	\]
	where $1+\cDes(T)\pmod n:=\{(1+i) \pmod n:\ i\in \cDes(T)\}$.
	A pair $(\cDes,p)$, which satisfies the first two axioms but not the third is called an {\em Escherian cyclic descent extension}. 
\end{definition}



\begin{example}
	Consider the symmetric group $\symm_n$ on $n$ letters
	and the standard {\em descent set} of a permutation $\pi = [\pi_1, \ldots, \pi_n]$, 
	\[
	\Des(\pi) := \{1 \le i \le n-1 \,:\, \pi_i > \pi_{i+1} \}
	\quad \subseteq [n-1].
	\]
	A corresponding {\em cyclic descent set} was defined by Cellini~\cite{Cellini} as
	\[
	\CDes(\pi) := \{1 \leq i \leq n \,:\, \pi_i > \pi_{i+1} \}
	\quad \subseteq [n],
	\]
	with the convention $\pi_{n+1}:=\pi_1$.  
	The pair $(\CDes,p)$, where $p:\symm_n\rightarrow \symm_n$ is the cyclic rotation defined by $p([\pi_1,\ldots,\pi_n]):=[\pi_n,\pi_1,\ldots,\pi_{n-1}]$, 
	is a cyclic descent extension for $\symm_n$.
\end{example}

\smallskip

Cyclic descent extensions were introduced in the study of Lie algebras~\cite{Klyachko} and 
descent algebras~\cite{Cellini}. 
Surprising connections of cyclic descent extensions to a variety of mathematical areas were found later.
For connections of cyclic descents to Kazhdan-Lusztig theory see~\cite{Rhoades}; 
for topological aspects and connections to the Steinberg torus see~\cite{DPS}; 
for twisted Sch\"utzenberger promotion see~\cite{Rhoades, Huang}; 
for cyclic quasisymmetric functions and Schur-positivity see~\cite{AGRR, HR19, BER}; 
for higher Lie characters see~\cite{HR19}; 	
and for Postnikov's toric Schur functions and Gromov-Witten invariants see~\cite{ARR}.

\smallskip

The question addressed in~\cite{HR19} was: 
which conjugacy classes in $\symm_n$ carry a cyclic descent extension?
Cellini's cyclic descent map does not provide a cyclic descent extension on most conjugacy classes.
However, it turns out that most conjugacy classes carry such an extension.

\begin{example}
	Consider the conjugacy class of transpositions in $\symm_4$
	\[
	\{2134, 3214, 4231, 1324, 1432, 1243\}.
	\]
	Cellini's cyclic descent sets are
	\[
	\{1,4\}, \{1,2,4\}, \{1,3\}, \{2,4\}, \{2,3,4\}, \{3,4\}
	\]
	respectively; and this collection is not closed under cyclic rotation.
	On the other hand, defining the cyclic descent sets by
	\[
	\cDes(2134)=\{1,4\},\ \cDes(3214)=\{1,2\},
	\  \cDes(4231)=\{1,3\}, 
	\]
	\[
	\cDes(1324)=\{2,4\},\ \cDes(1432)=\{2,3\},\ \cDes(1243)=\{3,4\}
	\]
	and the map $p$ by
	\[
	3214 \rightarrow 1432 \rightarrow 1243 \rightarrow 2134 \, (\rightarrow 3214),
	\qquad
	4231 \rightarrow 1324 \, (\rightarrow 4231)
	\]
	yields a pair $(\cDes,p)$ which is a cyclic descent extension for this conjugacy class.
\end{example}

A full characterization of the conjugacy classes in $\symm_n$ which carry a cyclic descent extension was 
given.
\begin{theorem}\cite[Theorem 1.4]{HR19}\label{thm:HR}
	A conjugacy class of permutations of cycle type $\lambda$ carries a cyclic descent extension if and only if 
	$\lambda$ is not equal to $(r^s)$ for any square-free integer $r$.    
\end{theorem}

The proof of Theorem~\ref{thm:HR}, presented in~\cite{HR19}, is algebraic and not constructive. 

\medskip

\begin{problem}~\cite[Problem 7.11]{HR19}
	Find an 
	explicit combinatorial description of the cyclic descent extension
	for conjugacy classes, whenever it exists.
\end{problem}
In this paper we present a solution of this problem
for the conjugacy classes of involutions. 
For $n \ge k \ge 0$ with $n-k$ even, 
let $\I_{n,k}$ be the conjugacy class of involutions with $k$ fixed points in the symmetric group $\symm_n$.   
We present a purely combinatorial constructive proof of the following result. 

\begin{theorem}\label{thm:main11}
	For every $n> k> 0$ with $n-k$ even, $\I_{n,k}$ carries a  cyclic descent extension.
	For $k =0$ and $k = n$ there is only an Escherian cyclic extension.
\end{theorem}



In order to construct an explicit 
cyclic descent extension for conjugacy classes of involutions with fixed points, we have to consider first the conjugacy class of fixed-point-free involutions. 
It will be shown that a certain geometrically-defined set-valued function on perfect matchings is equidistributed with the standard descent set
on fixed-point-free involutions, leading to
an Escherian cyclic descent extension for this conjugacy class of involutions and an ordinary (non-Escherian) extension for the rest.

\medskip

For $n \ge k \ge 0$ with $n-k$ even, 
let $\M_{n,k}$ be the set of partial matchings on $n$ points, labeled by $[n]:=\{1,\dots,n\}$, with exactly $k$ unmatched points.

\begin{remark}\label{rem:matchings-involutions}
	An involution $(i_1,i_2)\cdots(i_{r-1},i_r)\in\I_{n,k}$ 
	(with $n = 2r+k$) 
	may be naturally interpreted as 
	the matching $m\in\M_{n,k}$ with matched pairs $\{i_1,i_2\},\dots,\{i_{r-1},i_r\}$.  
	This interpretation will be used frequently. 
	Throughout this paper, the notations $\I_{n,k}$ and $\M_{n,k}$ are interchangeable. 
\end{remark}

\begin{definition}
	The {\em standard descent set} of a matching $m\in\M_{n,k}$, denoted $\Des(m)$, is defined via the one-line notation of the corresponding 
	involution in $\I_{n,k}$.
\end{definition}

\begin{definition}\label{def:chord_des0}
	The {\em geometric descent set} of a matching $m\in \M_{n,k}$, 
	denoted $\MDes(m)$,  consists of the geometric descents of $m$,
	defined as follows.
	Draw the $n$ points on a horizontal line in the real plane and label them by $1,\dots,n$ from left to right.
	Indicate a matched pair $\{i,j\}$, with $i < j$, by drawing an arc in the upper half plane from the point labeled $i$ to the point labeled $j$.
	An index $i\in [n-1]$ is a {\em geometric descent} of the matching $m\in \M_{n,k}$ if one of the following conditions holds:
	\begin{enumerate}
		\item
		$\{i,i+1\}$ is a matched pair in $m$.
		\item 
		The arc containing $i$ intersects the arc containing $i+1$.
		\item
		$i$ is unmatched and $i+1$ is matched.
	\end{enumerate}
\end{definition}

See Figure~\ref{fig:M1} for an example.

\begin{figure}[htb]
	\begin{center}
		\begin{tikzpicture}[scale=0.35]
			\draw (-3,0) node {$m=\ $};
			\draw[fill] (0,0) circle (.1); \draw(0,0) node[below]
			{1};
			\draw[fill] (2,0) circle (.1)node[below]
			{2};
			\draw[fill] (4,0) circle (.1); \draw(4,0) node[below]
			{3};
			\draw[fill] (6,0) circle (.1); \draw(6,0) node[below]
			{4};
			\draw[fill] (8,0) circle (.1) node[below]
			{5};
			\draw[fill] (10,0) circle (.1) node[below]
			{6};
			\draw[fill] (12,0) circle (.1) node[below]
			{7};
			\draw[fill] (14,0) circle (.1) node[below]  {8};
			
			
			\draw[red] (10,0) arc(0:180:5);
			\draw[red] (6,0) arc(0:180:1);
			\draw[red] (12,0) arc(0:180:2);
			%
			
			
		\end{tikzpicture}
	\end{center}
	\caption{$m=(1,6)(3,4)(5,7)\in \M_{8,2}$, has $\MDes(m)=\{2,3,5,6\}$ and  $\Des(m)=\Des([6,2,4,3,7,1,5,8])=\{1,3,5\}$.}\label{fig:M1}
\end{figure}


For a finite set of positive integers $J$ let ${\bf x}^J:=\prod\limits_{j\in J}x_j$. 

\begin{lemma}\label{thm:main1}
	For every $n\ge 0$ 
	\[
	\sum_{m \in \M_{2n,0}} {\bf x}^{\Des(m)} {\bf y}^{\MDes(m)}
	= \sum_{m \in \M_{2n,0}} {\bf x}^{\MDes(m)} {\bf y}^{\Des(m)}.
	\]
\end{lemma}


For a matching $m\in \M_{n,k}$ let $\crn(m)$ and $\nest(m)$ be the crossing number and nesting number of $m$, respectively; see Definition~\ref{defn:cr-nest} below.
Using Lemma~\ref{thm:main1} we will prove the following.  

\begin{theorem}\label{cor:main11}
	For every $n\ge k\ge 0$ with $n-k$ even 
	\[
	\sum\limits_{m\in \M_{n,k}}  q^{\crn(m)} 
	{\bf x}^{\MDes(m)}=\sum\limits_{m\in \M_{n,k}} q^{\nest(m)}
	{\bf x}^{\Des(m)}.
	\]
\end{theorem}

\smallskip

Bijective proofs of Lemma~\ref{thm:main1} and Theorem~\ref{cor:main11} 
will be described in Section~\ref{sec:bijections}.

\smallskip




Let $\M_n:=\sqcup_{k=0}^n \M_{n,k}$ be the set of all matchings on $n$ points, labeled by $1,\dots, n$. 
Let $\um(m)$ be the number of unmatched points in a matching $m\in \M_n$. 
For a partition $\lambda$ 
let $\height(\lambda)$ be the number of parts in $\lambda$, 
let
$\oddcol(\lambda)$ be the number of odd parts in the conjugate partition,
and let $s_\lambda$ be the corresponding Schur function. 
For $D\subseteq [n-1]$ let $F_{n,D}$ be the corresponding fundamental quasisymmetric function. For definitions and more details see Subsection~\ref{sec:Schur_background}. 
The following Schur-positivity phenomenon follows from the proof of Theorem~\ref{cor:main11}.


\begin{theorem}\label{thm:main0}
	For every $n\ge 0$
	\[
	\sum\limits_{m\in \M_{n}} q^{\um(m)}t^{\crn(m)} F_{n,\MDes(m)}
	= \sum\limits_{\lambda\vdash n}  q^{\oddcol(\lambda)} t^{\lfloor \height(\lambda)/2 \rfloor} s_\lambda. 
	\]
\end{theorem}

\bigskip

The existence of cyclic descent extensions, on conjugacy classes of involutions with fixed points and other combinatorial sets, follows. 
To verify this observe, first, that there is a very natural cyclic extension of $\MDes$ on $\M_{n,k}$.

\begin{definition}\label{def:chord_des}
	Draw $n$ points 
	on a circle and label them by $1,\dots,n$ clockwise.
	Indicate a matched pair by drawing a chord between the corresponding points.
	A point $i\in [n]$ is a {\em cyclic geometric descent} of a matching $m\in \M_{n,k}$ if one of the following conditions holds (where addition is modulo $n$):
	\begin{enumerate}
		\item
		$\{i,i+1\}$ is a 
		chord in $m$.
		\item 
		The chord containing $i$ intersects the chord containing $i+1$.  
		\item
		$i$ is unmatched and $i+1$ is matched.
	\end{enumerate}
	The {\em cyclic geometric descent set} of $m$ is denoted by $\cMDes(m)$. 
\end{definition}

See Figure~\ref{fig:M2} for an example.


\begin{figure}[htb]
	\begin{center}
		\begin{tikzpicture}[scale=0.35]
			\draw (-7,0) node {$m=\ $};
			\draw (0,0) circle (4);
			\draw[fill] (90:4) circle (.1); 
			\draw(90:4) node[above]
			{8};
			\draw[fill] (45:4) circle (.1) node[above right]
			{1};
			\draw[fill] (0:4) circle (.1) node[right]
			{2};
			\draw[fill] (-45:4) circle (.1) node[below right]
			{3};
			\draw[fill] (-90:4) circle (.1) node[below]
			{4};
			\draw[fill] (-135:4) circle (.1) node[below left]
			{5};
			\draw[fill] (180:4) circle (.1) node[left]
			{6};
			\draw[fill] (135:4) circle (.1) node[above left]
			{7};
			
			\draw[red](45:4)--(180:4);    
			\draw[red](-135:4)--(135:4);  
			\draw[red](-45:4)--(-90:4);   
			
			\draw (8,0) node {,\ \ \ $r(m)=\ $};
			\draw (16,0) circle (4);
			\draw[fill] (16,4) circle (.1) node[above]
			{8};
			\draw[fill] (18.8, 2.8) circle (.1) node[above right]
			{1};
			\draw[fill] (20, 0) circle (.1) node[right]    
			{2};
			\draw[fill] (18.8, -2.8) circle (.1) node[below right]
			{3};
			\draw[fill] (16, -4) circle (.1) node[below]
			{4};
			\draw[fill] (13.2, -2.8) circle (.1) node[below left]
			{5};
			\draw[fill] (12, 0) circle (.1) node[left]
			{6};
			\draw[fill] (13.2, 2.8) circle (.1) node[above left]
			{7};
			
			\draw[red](13.2, 2.8)--(20, 0);    
			\draw[red](12, 0)--(16, 4);  
			\draw[red](13.2, -2.8)--(16, -4);   
			
		\end{tikzpicture}
	\end{center}
	\caption{$m=(1,6)(3,4)(5,7)\in \M_{8,2}$ has $\MDes(m) = \{2,3,5,6\}$ 
		and  $\cMDes(m) = \{2,3,5,6,\red{8}\}$. 
		Rotating $m$ by $2\pi/8$ yields $r(m)=(2,7)(4,5)(6,8)$ with
		$\MDes(r(m)) = \cMDes(r(m)) = \{3,4,6,7,1\}$.}
	\label{fig:M2}
\end{figure}

The  proof  of Theorem~\ref{cor:main11} applies 
an explicit bijection  
$\hat\iota :\I_{n,k}
\rightarrow \M_{n,k}$ for any $n\ge k\ge 0$, 
which satisfies 
\[
\MDes(\hat\iota(\pi))=\Des(\pi)
\qquad (\forall \pi\in \I_{n,k}). 
\]
Using $\cMDes$,
define $\cDes : \I_{n,k} \to [n]$ by
\[
\cDes(\pi) := \cMDes(\hat\iota(\pi))
\qquad (\forall \pi\in \I_{n,k}). 
\]
Let $r:\M_{n,k} \to \M_{n,k}$
correspond to 
clockwise rotation by $2\pi/n$.
Here is an explicit version of Theorem~\ref{thm:main11}.

\begin{proposition}\label{cor:main}
	Assume that $n \ge k \ge 0$ with $n-k$ even.
	\begin{itemize}
		\item[(a)] 
		If $0 < k < n$
		then the pair $(\cDes, \hat\iota^{-1} \circ r \circ \hat\iota)$ is a (non-Escherian) cyclic extension of $\Des$ on $\I_{n,k}$.
		\item[(b)] 
		If $k=0$ or $k=n$
		then the above pair is an Escherian cyclic extension of $\Des$ on $\I_{n,k}$.  
	\end{itemize}
\end{proposition}


The cyclic descent extension from Proposition~\ref{cor:main} can be further refined to certain subsets of $\I_{n,k}$,
yielding a combinatorial
cyclic descent extension for sets of standard Young tableaux of bounded height with a given number of odd columns.
Letting the height be at most $2$ with all columns even, or height at most $3$ with no further restrictions, give explicit 
cyclic descent extensions for the sets of Dyck paths and Motzkin paths of fixed length, respectively. These cyclic extensions  
coincide with those  determined by Dennis White~\cite{Petersen_PR} and Bin Han~\cite{Han}. 


The structure of this paper is as follows.
Section~\ref{sec:prel} contains some preliminary background.
Section~\ref{sec:bijections} contains bijective proofs of the equidistribution results, Lemma~\ref{thm:main1} and Theorem~\ref{cor:main11}.
Section~\ref{sec:main0} contains a proof of the Schur-positivity result, Theorem~\ref{thm:main0}.
Section~\ref{sec:CDE} deals with cyclic descent extensions and proves Proposition~\ref{cor:main}, thus Theorem~\ref{thm:main11}.
Finally, Section~\ref{sec:Gessel} presents a non-constructive proof of a refinement of Theorem~\ref{cor:main11}, based on a very recent unpublished result of Gessel.

\section{Preliminaries}\label{sec:prel}

\subsection{Permutations and tableaux}


For $1 \le k \le n$ denote 
$[n]:=\{1,2,\dots,n\}$ and $[k,n]:=\{k,k+1,\dots,n\}$.
A partition of a positive integer $n$ is a sequence $\lambda=(\lambda_1,\dots,\lambda_t)$ of weakly decreasing positive integers 
whose sum is $n$. Denote $\lambda\vdash n$. 

Let $\symm_n$ denote the symmetric group consisting of all permutations of $[n]$.
A permutation $\pi \in \symm_n$ will be represented by the one-row notation
$\pi = [\pi_1,\dots,\pi_n]\in \symm_n$, where $\pi_i := \pi(i)$ $(i \in [n])$;
denote also $\Fix(\pi) := \{i\in [n] \,:\, \pi(i) = i\}$, the set of fixed points of $\pi$.
Recall that the {\em descent set} of a permutation $\pi\in\symm_n$ is
\[
\Des(\pi) := \{i \in [n-1] \,:\, \pi(i)>\pi(i+1)\}.
\]
Another important family of combinatorial objects for which there
is a well-studied notion of descent set is the set of standard Young tableaux (SYT). 
Let $\SYT(\lambda)$ denote the set of standard Young tableaux of
shape $\lambda$, where $\lambda$ is a  partition of $n$. 
We draw tableaux in English notation, as in Figure~\ref{fig:SYT}.
The {\em descent set} of $T \in \SYT(\lambda)$ is
\[
\Des(T) := \{i\in[n-1] \,:\, i+1 \text{ is in a lower row than $i$ in $T$}\}.
\]
For example, the descent set of the SYT in Figure~\ref{fig:SYT} is $\{1,3,5,6\}$.

\begin{figure}[htb]
	$$\young(1359,246,78)$$
	\caption{A SYT of shape $\lambda = (4,3,2)$.} 
	\label{fig:SYT}
\end{figure}

The Robinson-Schensted (RS) correspondence is a bijection $\pi \mapsto (P_\pi, Q_\pi)$ from permutations in $\symm_n$ to pairs of standard Young tableaux (SYT) of the same shape and size $n$. 
The common shape $\lambda$ of the {\em insertion tableau} $P_\pi$ and the {\em recording tableau} $Q_\pi$ is called the {\em shape} of the permutation $\pi$. 
We recall basic properties of the RS correspondence that will be used in the paper. For more details see, e.g., \cite{Sagan_book}. 

The {\em height} $\height(\lambda)$ of a shape $\lambda$ is the number of rows in $\lambda$.

\begin{proposition}\label{prop:RS_height}\cite{schensted}
	For every permutation $\pi \in \symm_n$, the height of the shape of $\pi$ is equal to the maximal length of a decreasing subsequence in the one-line notation of $\pi$.
\end{proposition}

\begin{proposition}\label{prop:RS_properies}\cite[Propositions 14.4.12 and 14.10.6]{AR15}
	\begin{itemize}
		\item[1.] 
		$P_\pi=Q_{\pi^{-1}}$, thus  $Q_\pi=P_{\pi^{-1}}$
		and $P_\pi=Q_\pi$ if and only if $\pi \in \symm_n$ is an involution.
		\item[2.] 
		$\Des(Q_\pi)=\Des(\pi)$ for all $\pi \in \symm_n$.
	\end{itemize}
\end{proposition}

\begin{proposition}\label{RS_involution_odd}~\cite{Schu}
	The number of columns of odd length in the shape of an involution with $k$ fixed points is equal to $k$. 
\end{proposition}

Let $\SYT_n$ denote the set of all SYT of size $n$,
and let $\SYT_{n,k}$ denote  
	the set of standard Young tableaux of size $n$ having $k$ columns of odd length.
	Consider the map $Q : \symm_n \to \SYT_n$ defined by mapping $\pi \in \symm_n$ to the corresponding Robinson-Schensted recording tableau $Q_\pi$.
	
	\begin{corollary}\label{cor:RS_involution}
		The map $Q$, restricted to involutions with a fixed number of fixed points, 
		is a descent-set-preserving bijection 
		from the set $\I_{n,k}$ of involutions in $\symm_n$ with $k$ fixed points 
		to the set $\SYT_{n,k}$ of standard Young tableaux of size $n$ with $k$ odd columns. 
	\end{corollary}
	
	\begin{proof}
		By Proposition~\ref{prop:RS_properies}.1, the map $Q$ determines a bijection from the set of involutions $\I_n$ to the set of all SYT of size $n$. 
		By Proposition~\ref{prop:RS_properies}.2, this map is descent-set-preserving and, by  Proposition~\ref{RS_involution_odd}, the pre-image of $\SYT_{n,k}$ is $\I_{n,k}$. \end{proof}
	
	
	Let $U$ and $V$ be disjoint finite totally-ordered sets of letters, and let $\sigma$ and $\tau$ be two permutations of $U$ and $V$, respectively. 
	The {\em shuffle} of $\sigma$ and $\tau$, denoted by $\sigma \shuffle \tau$,
	is the set of all permutations of the disjoint union $U\sqcup V$ in which the letters of $U$ appear in the same order as in $\sigma$ and the letters of $V$ appear in the same order as in $\tau$. 
	For sets
	$A$ and $B$ of permutations on disjoint finite totally-ordered sets of letters $U$ and $V$, respectively, denote by $A \shuffle B$ the set of all shuffles of a permutation in $A$ and a permutation in $B$.
	For example, if $A=\{12,21\}$ and $B=\{43\}$, then $A \shuffle B = \{1243,1423,1432,4123,4132,4312,2143,2413,2431,4213,4231,4321\}$.
	
	\begin{observation}\label{obs:RS}
		By the definition of the RS correspondence, 
		the smallest $n-k$ letters in $P_\pi$ 
		form a sub-tableau which depends only on their relative positions in $\pi$. 
	\end{observation}
	
	In particular, letting $\sigma$ be a permutation on $[k]$ and $\tau$ a permutation on $[n]\setminus [k]$, all $\pi\in \sigma \shuffle \tau$ have a common sub-tableau of $P_\pi$ consisting of the smallest $k$ letters. 
	
	\begin{proposition}\label{prop:RS_properies3}
		For every $\sigma\in \I_{n-k,0}$ and $\pi\in\sigma\shuffle [n-k+1,\dots,n]$, the number of odd columns in the RS shape of $\pi$ is equal to $k$.
	\end{proposition}
	
	\begin{proof}
		Since $\sigma\in \I_{n-k,0}$, 
		by Proposition~\ref{RS_involution_odd} all the columns 
		of its shape have even length.
		By Observation~\ref{obs:RS}, 
		the shape of the sub-tableau consisting of the smallest $n-k$ letters in $P_\pi$, which are the letters of $\sigma$, 
		is the shape of $\sigma$.
		On the other hand, for every shuffle $\pi\in \sigma\shuffle [n-k+1,\dots,n]$
		and every $n-k<i<n$,  
		$i\not\in \Des(\pi^{-1})$. 
		By Proposition~\ref{prop:RS_properies} this implies that $i\not\in \Des(P_\pi)$ for all such $i$, 
		so that the largest $k$ letters in $P_\pi$
		belong to distinct columns and increase from left to right.
		They are therefore in the bottom cells of the odd columns of $P_\pi$, and the result follows.  
	\end{proof}
	
	
	The proof of Proposition~\ref{prop:RS_properies3} implies the following.
	
	\begin{corollary}\label{cor:RS_properties3b}
		For every $\sigma\in \I_{n-k,0}$ and $\pi\in\sigma\shuffle [n-k+1,\dots,n]$, 
		the largest $k$ letters in $P_\pi$ 
		appear in the bottom cells of the $k$ odd columns of $P_\pi$, and they are increasing from left to right. 
	\end{corollary}
	
	\subsection{Involutions and oscillating tableaux}\label{sec:osc}
	
	
	Consider the Young lattice whose elements are all partitions, ordered by inclusion of the corresponding Young 
	diagrams. A standard Young tableau of shape $\lambda$ may be viewed as a maximal chain, in the Young lattice, from the empty partition to $\lambda$; see, e.g., \cite[\S 14.2.5.1]{AR15}.
	A variation of this description yields {\em oscillating tableaux}, which correspond to general paths in the Hasse diagram of the Young lattice, from the empty diagram to a diagram of shape $\lambda$. 
	The {\em size} of the oscillating tableau is the length of the path, and its {\em shape} is $\lambda$.
	We focus on closed paths of length $2n$ from the empty diagram to itself; in other words, on oscillating tableaux of size $2n$ with an empty shape. The set of all such oscillating tableaux will be denoted by $\O_{2n}$. 
	A key tool in this paper is {\em Sundaram's bijection} $s:\I_{2n,0} \to \O_{2n}$, 
	from the set $\I_{2n,0}$ of fixed-point-free involutions in $S_{2n}$ to the set $\O_{2n}$ of oscillating tableaux of size $2n$ and empty shape;
	see~\cite{Sundaram}. We hereby describe this bijection.
	
	
	\begin{definition}\label{def:s}({
			Sundaram's bijection}~\cite{Sundaram})
		Let $\pi \in \I_{2n,0}$.
		We start with $\lambda^0=\varnothing$.
		For $1\le d\le 2n$, define a standard Young tableau of shape $\lambda^d$, with letters forming a subset of $[2n]$, from a presumably-defined standard Young tableau of shape $\lambda^{d-1}$, as follows.
		Let $t_d =(i,j)$, $i<j$, be the unique transposition which affects $d$ in the factorization of $\pi$ into a product of $n$ disjoint transpositions.
		If $d=i$, insert $j$ into the tableau of shape $\lambda^{d-1}$ using Robinson-Schensted insertion and get a tableau of shape $\lambda^d$. 
		If $d=j$, delete $j$ from the tableau of shape $\lambda^{d-1}$ and apply jeu-de-taquin to get a tableau of shape $\lambda^d$.
		We get a sequence of $2n+1$ tableaux of shapes $\lambda^d$, $0 \le d\le 2n$.
		Ignoring the letters in the tableaux yields a sequence of shapes, which is the oscillating tableau corresponding to $\pi$. 
	\end{definition}

	\begin{example}\label{ex:1}
		Let $\pi=(1,5)(2,4)(3,8)(6,7)\in \I_{8,0}$. The corresponding sequence of tableaux is
		\[
		\varnothing\ ,\ \young(5)\ ,\ \young(4,5)\ ,\ \young(48,5)\ ,\ \young(58) \ ,\  \young(8)\ ,\ \young(7,8)\ ,\ \young(8)\ ,\ \varnothing
		\]
		Thus, the oscillating tableau corresponding to $\pi$ is
		$$s(\pi)=(\varnothing\ ,\ \young(\hfil)\ ,\ \young(\hfil,\hfil)\ ,\ \young(\hfil\hfil,\hfil)\ ,\ \young(\hfil\hfil)\ ,\ \young(\hfil)\ ,\ \young(\hfil,\hfil)\ ,\ \young(\hfil)\ ,\ \varnothing)\in \O_{8}.$$
	\end{example}
	
	A special case of~\cite[Theorem 5.3]{Sundaram} is the following.

	\begin{theorem}
		The map $s:\I_{2n,0}\rightarrow \O_{2n}$ defined above is a bijection.
	\end{theorem}
	
	A characterization of the descents of $\pi$ in the language of oscillating tableaux follows.
	
	\begin{observation}\label{obs:Kim} [Kim, Proof of Theorem 3.4]
		For every $\pi\in \I_{2n,0}$, $i\in \Des(\pi)$ if and only if
		what we do in the $i^{th}$ and $(i+1)^{st}$ steps of the corresponding oscillating tableau $s(\pi)$ is either
		\begin{enumerate}
			\item 
			add a box in the $i^{th}$ step and then delete a box in the next step; or 
			\item 
			add a box in the $i^{th}$ step and then add another box in a  strictly lower row 
			in the next step; or
			\item 
			delete a box in the $i^{th}$ step and then delete another box in a strictly higher 
			row in the next step. 
		\end{enumerate}
		In all other cases, $i\not\in \Des(\pi)$.
	\end{observation}
	
	\begin{definition}\label{def:tr}
		For an oscillating tableau $O=(D_0, D_1, \ldots, D_{2n}) \in O_{2n}$, define the {\em transpose} (or {\em conjugate}) ${\rm tr}\ O := (D'_0, D'_1, \ldots, D'_{2n})$,
		where for each $0\le i\le 2n$, the diagram $D'_i$ is the transpose of the diagram $D_i$.
	\end{definition}
	
	
	
	Another bijection $t:\I_{2n,0}\to \O_{2n}$,
	using growth diagrams, 
	was described by Roby~\cite[\S 4.2]{Roby}.
	This bijection applies a growth diagram algorithm to (a half of) the permutation matrix corresponding to a fixed-point-free involution $\pi\in \I_{2n,0}$, with empty boundary conditions, and reads an oscillating tableau $t(\pi)$ from the main diagonal. It relates to Sundaram's bijection via conjugation.
	
	\begin{proposition}\label{prop:Roby}\cite[p. 69]{Roby}
		For every $\pi\in \I_{2n,0}$
		\[
		t(\pi)={\rm tr} (s(\pi)).
		\]
	\end{proposition}
	
	Denote the longest permutation in $S_{2n}$ by $w_0:=(1,2n)(2,2n-1)\cdots (n,n+1)$. 
	
	\begin{corollary}\label{cor:Roby}
		For every $\pi\in \I_{2n,0}$, 
		the oscillating tableau $s(w_0 \pi w_0)$ is the reverse of $s(\pi)$. 
	\end{corollary}
	
	\begin{proof}
		Conjugating a permutation $\pi\in S_{2n}$ by $w_0$ corresponds to reflecting its permutation matrix about its vertical midline as well as about its horizontal midline. 
		This is equivalent to a 180-degree rotation of the permutation matrix.  
		An inspection of the algorithm shows that for $\pi\in \I_{2n,0}$ this simply reverses the oscillating tableau $t(\pi)$ on the main diagonal. Combining this with Proposition~\ref{prop:Roby} completes the proof. 
	\end{proof}
	
	\subsection{Matchings}\label{sec:match}
	
	Chen et al.~\cite{Chen} generalized Sundaram's bijection, described in Subsection~\ref{sec:osc} above, 
	and applied it to the enumeration of crossings and nestings in perfect matchings and partitions.
	
	\begin{definition}\label{defn:cr-nest}
		Let $m \in \M_n$ be a matching.
		\begin{itemize}
			\item[1.] 
			The {\em crossing number} $\crn(m)$ of 
			$m$ is the maximal $r$ such that there exist 
			matched pairs $\{i_1,j_1\}, \{i_2,j_2\}, \ldots, \{i_r,j_r\}$ in $m$ with $1 \le i_1 < \cdots < i_r < j_1 < \cdots <j_r \le n$.
			\item[2.] 
			The {\em nesting number} $\nest(m)$ of
			$m$ is the maximal $r$ such that there exist 
			matched pairs $\{i_1,j_1\}, \{i_2,j_2\}, \ldots, \{i_r,j_r\}$ in $m$ with $1 \le i_1 < \cdots < i_r < j_r < \cdots <j_1 \le n$.
		\end{itemize}
	\end{definition}
	
	\begin{example}
		For 
		$m\in \M_{8,2}$
		as in Figure~\ref{fig:M1},
		$\{1,6\},\{5,7\}$ is a maximal crossing, thus $\crn(m) = 2$.
		Also, 
		$\{1,6\},\{3,4\}$ is a maximal nesting, thus $\nest(m) = 2$.
	\end{example}
	
	
	Chen et al.\ introduced the involution $\iota: \M_{2n,0} \to \M_{2n,0}$, defined by
	\[
	\iota := s^{-1} \circ {\rm tr} \circ s 
	\]
	Here $s$ is Sundaram's bijection, $s: \I_{2n,0} \to \O_{2n}$, described in Definition~\ref{def:s}, 
	and {\rm tr} is the transpose operation on  oscillating tableaux, as in Definition~\ref{def:tr}.
	Following Remark~\ref{rem:matchings-involutions}, 
	we identify perfect matchings in $\M_{2n,0}$ with involutions in $\I_{2n,0}$.
	
	\begin{example}
		Let $\pi=(1,5)(2,4)(3,8)(6,7) \in \I_{8,0}$ as in Example~\ref{ex:1}.
		Then
		\[
		{\rm tr} \circ s (\pi)= 
		(\varnothing\ ,\ \young(\hfil)\ ,\ \young(\hfil\hfil)\ ,\ \young(\hfil\hfil,\hfil)\ ,\ \young(\hfil,\hfil)\ ,\ \young(\hfil)\ ,\ \young(\hfil\hfil)\ ,\ \young(\hfil)\ ,\ \varnothing)
		\]
		and $\iota(\pi)=s^{-1}\circ {\rm tr} \circ s (\pi)= (1,4)(2,7)(3,5)(6,8) \in \I_{8,0}$.
	\end{example}
	
	\begin{theorem}\cite{Chen}\label{thm:Chen}
		For every $m\in \M_{2n,0}$
		\[
		\crn(m) = \nest(\iota(m)).
		\]
		Thus
		\[
		\sum_{m \in \M_{2n,0}} q^{\crn(m)} t^{\nest(m)} =
		\sum_{m \in \M_{2n,0}} q^{\nest(m)} t^{\crn(m)}.
		\]
	\end{theorem}

	\section{Geometric versus standard descents: equidistribution results}\label{sec:bijections}
	
	
	
	The bijection of Chen et al., presented in Subsection~\ref{sec:match}, is applied 
	in Subsection~\ref{sec:main1} to prove 
	Lemma~\ref{thm:main1}. This bijection serves as a component in the proof, in Subsection~\ref{sec:main11}, 
	of the following result.
	
	\begin{theorem}\label{prop:f_strong}
		For every $n\ge k\ge 0$ there exists an explicit bijection  
		$\hat\iota :\M_{n,k}
		\rightarrow \M_{n,k}$, to be described in Definition~\ref{def:extended_iota}, 
		which satisfies 
		\[
		\MDes(\hat\iota(m))=\Des(m)\ \ \ {\rm{and}}  \quad  \nest(\hat\iota(m))=\crn(m) \qquad (\forall m\in \M_{n,k}). 
		\]
	\end{theorem}
	
	Theorem~\ref{cor:main11} follows. 
	The 
	bijection $\hat\iota :\M_{n,k}
	\rightarrow \M_{n,k}$ 
	is 
	used to prove Theorem~\ref{thm:main0} in Subsection~\ref{sec:proof_main0},
	and to determine cyclic descents on involutions in Section~\ref{sec:CDE}. 

	\subsection{Proof of Lemma~\ref{thm:main1}}\label{sec:main1} 
	
	Recall the 
	involution $\iota: \M_{2n,0}\rightarrow \M_{2n,0}$ introduced by  
	Chen et al.~\cite{Chen}, 
	described in Subsection~\ref{sec:match}. 
	
	\begin{proposition}\label{t:111}
		The involution  
		$\iota :\M_{2n,0} \rightarrow \M_{2n,0}$ satisfies
		\[
		\Des(\iota(m))=\MDes(m)
		\qquad (\forall m\in \M_{2n,0}). 
		\]
	\end{proposition}
	
	\begin{proof}
		Consider $m \in \M_{2n,0}$ as a fixed-point-free involution $\pi\in\I_{2n,0}$ (see  Remark~\ref{rem:matchings-involutions}).   
		The oscillating tableau $s(\hat \pi)$, corresponding to the involution $\hat\pi:=\iota(\pi)$, is the conjugate of the oscillating tableau $s(\pi)$: 
		$s(\hat\pi) = {\rm tr}(s(\pi))$. 
		Here $s:\I_{2n,0}\rightarrow \O_{2n}$ is Sundaram's bijection, described in 
		Definition~\ref{def:s},   
		and {\rm tr} is the conjugation operation on  oscillating tableaux, as in Definition~\ref{def:tr}.
		
		We will 
		show that $\Des(\hat\pi)=\MDes(\pi)$.
		
		Fix $1 \le i < 2n$. There are seven possible cases.
		
		\begin{enumerate}
			\item 
			$(i,i+1)$ is a chord in $\pi$.
			\item 
			there exist $a<i$ and $b>i+1$, such that $(a,i)$ and $(i+1,b)$ are chords in $\pi$.
			\item 
			there exist $a<i$ and $b>i+1$, such that $(i,b)$ and $(a,i+1)$ are chords in $\pi$.
			\item 
			there exist $i+1<a<b$ such that $(i,a)$ and $(i+1,b)$ are chords in $\pi$.
			\item 
			there exist $i+1<a<b$ such that $(i+1,a)$ and $(i,b)$ are chords in $\pi$.
			\item 
			there exist $a<b<i$ such that $(a,i)$ and $(b,i+1)$ are chords in $\pi$.
			\item 
			there exist $a<b<i$ such that $(a,i+1)$ and $(b,i)$ are chords in $\pi$.
		\end{enumerate}
		
		By Definition~\ref{def:chord_des0},
		$i\in \MDes(\pi)$ in cases (1), (3), (4) and (6)
		and $i\not\in \MDes(\pi)$ in all other cases.
		
		By Definition~\ref{def:s} of an oscillating tableau 
		and basic properties of the insertion algorithm,  
		what we do in the $i^{th}$ and $(i+1)^{st}$ steps of the first $5$ cases above is
		\begin{enumerate}
			\item 
			add a box and then delete a box.
			\item 
			delete a box and then add a box.
			\item 
			add a box and then delete a box. 
			\item 
			add a box and then add another box in a weakly higher row.
			\item 
			add a box and then add another box in a strictly lower row.
		\end{enumerate}
		Cases (6) and (7) require more subtle analysis.
		Consider the involutions $\pi$ and $\pi' := w_0 \pi w_0$, and denote $i' := 2n-i$ (so that $i'+1 = 2n+1-i$), $a' := 2n+1-b$ and $b' := 2n+1-a$.
		Then case (6) for $\pi$ translates into
		\begin{itemize}
			\item[(6')] 
			there exist $b' > a' > i'+1$ such that $(i'+1,b')$ and $(i',a')$ are chords in $\pi'$,
		\end{itemize}
		namely case (4) for $\pi'$.
		Similarly, case (7) for $\pi$ translates into
		\begin{itemize}
			\item[(7')] 
			there exist $b' > a' > i'+1$ such that $(i'+1,a')$ and $(i',b')$ are chords in $\pi'$,
		\end{itemize}
		namely case (5) for $\pi'$.
		By Corollary~\ref{cor:Roby}, $s(\pi')$ is the reverse of $s(\pi)$. We conclude that 
		what we do in the $i^{th}$ and $(i+1)^{st}$ steps of cases (6) and (7) for $\pi$ is 
		\begin{itemize}
			\item[(6)] 
			delete a box and then delete another box in a weakly lower row.
			\item[(7)] 
			delete a box and then delete another box in a strictly higher row.
		\end{itemize}
		
		
		\smallskip
		
		
		This is the description for $\pi$. For $\hat\pi = \iota(\pi)$ we have a conjugate oscillating tableau. 
		The description for cases (1)-(3) remains the same,
		whereas case (4) is switched with case (5) and
		case (6) is switched with case (7). 
		By Observation~\ref{obs:Kim}, this translates to $i\in \Des(\hat\pi)$ in cases (1), (3), (4) and (6), but not in the other cases.
		This completes the proof.
	\end{proof}
	
	\begin{remark}
		Arguments, similar to those used in the proof of Lemma~\ref{thm:main1}, 
		were  used by 
		Kim~\cite{Kim} to prove the symmetry of the Eulerian and Mahonian distributions on $\I_{2n,0}$.
	\end{remark}
	
	
	
	
	The following refinement of Lemma~\ref{thm:main1} follows.
	
	\begin{corollary}\label{cor:main1}
		For every $n\ge 0$
		\[
		\sum\limits_{m\in \M_{2n,0}} {\bf x}^{\MDes(m)} {\bf y}^{\Des(m)} q^{\crn(m)} t^{\nest(m)} 
		= \sum\limits_{m\in \M_{2n,0}} {\bf x}^{\Des(m)} {\bf y}^{\MDes(m)} q^{\nest(m)} t^{\crn(m)}.
		\]
	\end{corollary}
	
	\begin{proof}
		By Proposition~\ref{t:111}, 
		the involution $\iota :\M_{2n,0} \to \M_{2n,0}$ satisfies $\Des(\iota(m)) = \MDes(m)$ for all $m \in \M_{2n,0}$. 
		By Theorem~\ref{thm:Chen}, 
		$\nest(\iota(m))=\crn(m)$. Since $\iota$ is an involution, 
		$\Des(m)=\MDes(\iota(m))$ and $\nest(m)=\crn(\iota(m))$. 
		Thus
		\[
		\sum_{m \in \M_{2n,0}} {\bf x}^{\MDes(m)} {\bf y}^{\Des(m)} q^{\crn(m)} t^{\nest(m)} 
		= \sum_{m' \in \M_{2n,0}} {\bf x}^{\Des(m')} {\bf y}^{\MDes(m')} q^{\nest(m')} t^{\crn(m')}, 
		\]
		where $m' :=\iota(m)$. 
	\end{proof}
	
	\subsection{Proof of Theorem~\ref{cor:main11}}\label{sec:main11} 
	
	In this subsection we describe a map 
	\[
	\hat\iota:\M_{n,k}\rightarrow \M_{n,k},
	\]
	for any $0\le k\le n$, which generalizes the bijection $\iota:\M_{2n,0}\rightarrow \M_{2n,0}$ used in the previous subsection.   
	It will be shown that $\hat\iota$ is a bijection which maps the descent set to the geometric descent set and the crossing number to the nesting number,  
	implying Theorem~\ref{cor:main11}.
	
	\medskip
	
	Recall that $\M_{n,k}$  is naturally identified with $\I_{n,k}$  (Remark~\ref{rem:matchings-involutions}). In the rest of this section it will be more convenient to consider involutions in $\I_{n,k}$, rather than matchings in $\M_{n,k}$, since the shuffle operation and the RS correspondence used here are defined in terms of permutations (in particular, involutions). 
	
	\begin{remark}
		The bijection of Chen et al.\ is defined for  involutions with fixed points as well. It is an involution which maps the crossing number to the nesting number and 
		preserves the fixed point set. 
		Unfortunately, for involutions with fixed points it does not map $\MDes$ to $\Des$ and vice versa. For example, Chen et al.'s involution maps
		$\pi = (1,4)(2,5)(3)$ to $\sigma = (1,5)(2,4)(3)$, 
		but $\Des(\pi) = \{2,3\} \ne \MDes(\sigma) = \{3\}$ and also $\Des(\sigma) = \{1,2,3,4\} \ne \MDes(\pi) = \{1,3,4\}$.
	\end{remark}
	
	\begin{definition}\label{def:extended_iota}
		Fix $n \ge k \ge 0$, with $n-k$ even.
		\begin{itemize}
			\item[1.] 
			For every $\pi\in \I_{n,k}$, let $\rm{res}(\pi)$ be the pair $(\Fix(\pi), \sigma)$, where 
			$\Fix(\pi)$ is the set of fixed points of $\pi$,
			and $\sigma$ is the fixed-point-free involution in $\symm_{n-k}$ with the same relative order as that of $\pi$ on $[n]\setminus \Fix(\pi)$.
			\item[2.] 
			For $(J,\sigma)\in \binom{[n]}{k}\times \I_{n-k,0}$ 
			let $\rm{emb}(\sigma, J)$ be the permutation in 
			the set of all shuffles 
			$\I_{n-k,0}\shuffle [n-k+1,n-k+2,\dots n]$, 
			for which the letters in $[n-k]$ 
			are ordered as in $\sigma$, 
			and set of positions of the increasing subsequence 
			$[n-k+1,\dots n]$ is equal to $J$.
			\item[3.] 
			Define $\varphi: \I_{n,k} \longrightarrow  \I_{n-k,0}\shuffle [n-k+1,\dots,n]$ by 
			\[
			\varphi:\ \I_{n,k} 
			\overset{\rm{res}}{\longrightarrow} 
			\binom{[n]}{k}\times \I_{n-k,0} 
			\overset{(id,\iota)}{\longrightarrow}
			\binom{[n]}{k}\times \I_{n-k,0} \overset{{\rm{emb}}}{\longrightarrow}
			\I_{n-k,0}\shuffle [n-k+1,\dots,n],
			\]
			where $(id,\iota)(J,\sigma):=(J,\iota(\sigma))$.
			
			
			\item[4.] 
			For $\tau \in \I_{n-k,0} \shuffle [n-k+1,\dots,n]$ let 
			$q(\tau) \in \I_n$ be the RS preimage of $(Q_\tau,Q_\tau)$, where $Q_\tau$ is the recording tableau of $\tau$:
			\[
			q:\ 
			\I_{n-k,0}\shuffle [n-k+1,\dots,n] 
			\overset{Q}{\longrightarrow} 
			\SYT_{n}  
			\overset{{\rm{diag}}}{\longrightarrow} \SYT_{n}\times \SYT_{n}
			\overset{{\rm{RS}}^{-1}}{\longrightarrow}
			\I_{n},
			\]
			where $\SYT_{n}$ denotes the set of standard Young tableaux of size $n$ and $\I_n$ is the set of involutions in $\symm_n$.
			Note that $q(\tau)$ is an involution by Proposition~\ref{prop:RS_properies}.1. 
			
			\item[5.] 
			Let $\hat\iota := q \circ \varphi$.
			
		\end{itemize}
	\end{definition}
	
	The following proposition implies 
	Theorem~\ref{prop:f_strong}. 
	
	\begin{proposition}\label{prop:bijection}
		The map 
		$\hat\iota: \I_{n,k} \longrightarrow \I_{n}$ is a bijection from $\I_{n,k}$ onto itself, which satisfies 
		\[
		\MDes(\pi)=\Des(\hat\iota(\pi))\ \ \ {\rm{and}} \quad 
		\crn(\pi)=\nest(\hat\iota(\pi))\qquad (\forall \pi\in \I_{n,k}).
		\]
	\end{proposition}
	
	\begin{example} Let $\pi=[4,2,6,1,5,3]=(1,4)(3,6)(2)(5)\in \I_{6,2}$. Then 
		\[
		\varphi:\ [4,2,6,1,5,3]
		\overset{\rm{res}}{\longmapsto} 
		(\{2,5\} \,,\, (1,3)(2,4)) 
		\overset{(id,\iota)}{\longmapsto}
		(\{2,5\} \,,\, (1,4)(2,3))  \overset{{\rm{emb}}}{\longmapsto}
		[4,5,3,2,6,1], 
		\]
		and thus 
		\[
		\begin{aligned}
			\hat\iota:\ [4,2,6,1,5,3]
			&\overset{\varphi}{\longmapsto}
			[4,5,3,2,6,1]
			\overset{Q}{\longmapsto} 
			\young(125,3,4,6) \\
			&\overset{{\rm{diag}}}{\longmapsto} \left(\quad \young(125,3,4,6) \quad,\quad \young(125,3,4,6) \quad\right)
			\overset{{\rm{RS}}^{-1}}{\longmapsto}
			[1,6,4,3,5,2].
		\end{aligned}
		\]
		Namely, $\hat\iota(\pi)=(2,6)(3,4)(1)(5)\in \I_{6,2}$.   
		Indeed,  
		$\MDes(\pi)=\Des(\hat\iota(\pi))=\{2,3,5\}$ and $\crn(\pi)=\nest(\hat\iota(\pi)) = 2$.
	\end{example}

	To prove Proposition~\ref{prop:bijection} we first generalize 
	the concepts of crossing and nesting numbers,
	from matchings (or, equivalently, involutions) to shuffles of fixed-point-free involutions with increasing sequences. 
	
	\begin{definition}\label{def:nest_shuffles}
		For every $\tau \in \I_{n-k,0}\shuffle [n-k+1,\dots,n]$ define $\nest(\tau) := \nest(\sigma)$ and $\crn(\tau) := \crn(\sigma)$, 
		where $\sigma\in \I_{n-k,0}$ is obtained by deleting the letters $n-k+1,\dots,n$ from $\tau$.
	\end{definition}
	
	
	\begin{example}
		Let $\tau = [3,4,5,1,6,2] \in \I_{4,0}\shuffle [5,6]$. Then $\tau$ is not an involution,
		but deleting the letters $5$ and $6$ from $\tau$ gives a fixed-point-free involution
		$\sigma = [3,4,1,2]=(1,3)(2,4) \in \I_{4,0}$. 
		By definition,    
		$\crn(\tau) = \crn(\sigma) = 2$ and $\nest(\tau) = \nest(\sigma) = 1$.
	\end{example}
	
	\begin{lemma}\label{lem:shuffles}
		For any $n \ge k \ge 0$ with $n-k$ even,
		the map 
		\[
		\varphi: \I_{n,k} \rightarrow \I_{n-k,0} \shuffle [n-k+1,n-k+2,\dots n]
		\]
		is a bijection which satisfies 
		\[
		\MDes(\pi) = \Des(\varphi(\pi)) 
		\qquad (\forall \pi\in \I_{n,k})
		\]
		as well as
		\[
		\nest(\pi)=\crn(\varphi(\pi)) 
		\ \ \ {\rm{and}} \quad 
		\crn(\pi)=\nest(\varphi(\pi)) 
		\qquad (\forall \pi\in \I_{n,k}).    
		\]
	\end{lemma}
	
	\begin{proof} 
		By Definition~\ref{def:extended_iota},
		$\varphi$ is a bijection. 
		To show that it maps $\MDes$ to $\Des$,  
		let $1 \le i \le n$ and consider the four possible cases.
		
		\begin{description}
			\item[Case 1] 
			If $i,i+1 \in \Fix(\pi)$ then, 
			by Definition~\ref{def:chord_des0}, $i\not\in \MDes(\pi)$;
			and, by Definition~\ref{def:extended_iota}, $\varphi(\pi)(i) < \varphi(\pi)(i+1)$.
			\item[Case 2] 
			If $i\not\in\Fix(\pi)$ and $i+1\in \Fix(\pi)$ then 
			$i\not\in \MDes(\pi)$ and $\varphi(\pi)(i)<\varphi(\pi)(i+1)$.
			\item[Case 3] 
			If $i\in \Fix(\pi)$ and $i+1\not\in\Fix(\pi)$
			then $i\in \MDes(\pi)$ and $\varphi(\pi)(i)>\varphi(\pi)(i+1)$.
			\item[Case 4] 
			If $i,i+1\not\in\Fix(\pi)$ then, 
			by Definition~\ref{def:extended_iota}, we can ignore the fixed points and apply Proposition~\ref{t:111}, which shows that 
			$i\in \MDes(\pi)\Longleftrightarrow i \in \Des(\varphi(\pi))$. 
		\end{description}
		This proves the claim regarding $\MDes$ and $\Des$.
		The claim about crossing and nesting numbers follows from Definition~\ref{def:extended_iota}, Definition~\ref{def:nest_shuffles} and Theorem~\ref{thm:Chen}.
	\end{proof}
	
	\begin{corollary}
		\[
		\sum\limits_{\pi\in \I_{n,k}} {\bf x}^{\MDes(\pi)} q^{\crn(\pi)} t^{\nest(m)}
		= \sum\limits_{\pi\in \I_{n-k,0} \shuffle [n-k+1,\dots,n]} {\bf x}^{\Des(\pi)} q^{\nest(\pi)} t^{\crn(\pi)}.
		\]
	\end{corollary}
	
	\medskip

	
	
	
	
	
	
	To prove Proposition~\ref{prop:bijection} we also need the following lemmas.
	
	\begin{lemma}\label{lem:1}
		For every involution $\pi\in \symm_n$
		\[
		\nest(\pi) = \lfloor \height(Q_\pi) / 2 \rfloor, 
		\]
		where $Q_\pi$ is the RS  recording tableau of $\pi$.
	\end{lemma}
	
	
	\begin{proof}
		By Proposition~\ref{prop:RS_height},
		$\height(Q_\pi)$ is the length of the longest decreasing subsequence in the one-line notation of $\pi$.
		By Definition~\ref{defn:cr-nest},
		if $\nest(\pi) = r$ then there exists a sequence $1 \le i_1 < \cdots < i_r < i_{r+1} < \cdots < i_{2r} \le n$ such that, for every $1 \le j \le 2r$, $\pi(i_j)=i_{2r+1-j}$.
		Then $(i_{2r},\dots,i_1)$ is a decreasing subsequence in the one-line notation of $\pi$, so that
		\[
		\height(Q_\pi) \ge 2r = 2\nest(\pi).
		\]
		On the other hand, 
		the fixed points of $\pi$ form an increasing subsequence of its one-line notation; thus any decreasing subsequence contains at most one fixed point. 
		Let $(\pi(i_1),\ldots, \pi(i_t))$ be a decreasing subsequence of maximal length in the one-line notation of $\pi$. 
		Assume, first, that it contains no fixed points.
		Let $s := \max \{j:\ \pi(i_j) > i_j\}$. 
		If $s > t/2$ then $(\pi(i_1), \ldots, \pi(i_s), i_s, \ldots, i_1)$ is a decreasing subsequence of length $2s > t$ in $\pi$, contradicting the maximality of $t$.
		Similarly, if $s < t/2$ then $(i_t, i_{t-1}, \ldots, i_{s+1}, \pi(i_{s+1}), \ldots, \pi(i_t))$ is a decreasing subsequence of length $2(t-s) > t$ in $\pi$, contradicting the maximality of $t$.
		We deduce that $s = t/2$.
		The sequence $(\pi(i_1), \ldots, \pi(i_s), i_s, \ldots, i_1)$ is a decreasing sequence of maximal length in $\pi$, and corresponds to a nesting. Thus
		\[
		2\nest(\pi) \ge 2s = t = \height(Q_\pi).
		\]
		
		Finally, assume that the chosen decreasing subsequence $(\pi(i_1),\ldots, \pi(i_t))$ of maximal length in the one-line notation of $\pi$ contains a fixed point, say $\pi(i_s) = i_s$. Then for all $1 \le j \le t$,
		$\pi(i_j) > i_j$ if and only if $j < s$.  
		By the above argument, if either $2(s-1)+1 > t$ or $2(t-s)+1 > t$ then
		one can define a decreasing subsequence of length exceeding $t$, contradicting the maximality of $t$.  
		Thus $t+1 \le 2s \le t+1$, namely $2s = t+1$.
		The sequence $(\pi(i_1),\ldots,\pi(i_{s-1}), i_{s-1},\ldots,i_1)$ is decreasing subsequence corresponding to a nesting, so that
		\[
		2\nest(\pi) \ge 2(s-1) = t-1 = \height(Q_\pi)-1.
		\]
		This completes the proof.
	\end{proof}
	
	\begin{remark}\label{rem:1}
		By Proposition~\ref{RS_involution_odd}, $\height(Q_\pi)$ is even for every fixed-point-free involution $\pi\in \I_{2n,0}$. Hence, 
		for fixed-point-free involutions, no floors are required in Lemma~\ref{lem:1}, i.e., 
		\[
		\nest(\pi) = \height(Q_\pi) / 2 
		\qquad (\forall \pi\in \I_{2n,0}).
		\]
	\end{remark}
	
	\begin{corollary}\label{t.tau}
		If $\sigma \in \I_{n-k,0}$ and $\tau \in \sigma \shuffle [n-k+1,\ldots,n]$ then
		\[
		\nest(\tau) = \lfloor \height(Q_\tau) / 2 \rfloor.
		\]
	\end{corollary}
	
	\begin{proof}
		For $\pi\in \symm_n$ let $\ell(\pi)$ be the length of the longest decreasing subsequence in $\pi$.
		Observe that $\ell(\tau)-\ell(\sigma)\in\{0,1\}$;
		hence, by Proposition~\ref{prop:RS_height}, 
		\[
		\height(Q_\tau)-\height(Q_\sigma)\in\{0,1\}.
		\]
		By Definition~\ref{def:nest_shuffles} and Remark~\ref{rem:1} 
		we deduce
		\[
		\nest(\tau)=\nest(\sigma)= \height(Q_\sigma) / 2 = \lfloor \height(Q_\tau) / 2 \rfloor.
		\]
	\end{proof}
	
	
		Recall the map $Q: \symm_n \to \SYT_n$ sending each $\pi\in \symm_n$ to the corresponding RS recording tableau $Q_\pi$.
		Recall also the notation $\SYT_{n,k}$ for the set of all SYT of size $n$ with $k$ odd columns.
		
		\begin{lemma}\label{lem:Q}
			For any $n \ge k \ge 0$ with $n-k$ even,
			the map $Q$ restricts to a descent-set-preserving bijection from the set of shuffles $\I_{n-k,0} \shuffle [n-k+1,\dots,n]$ to $\SYT_{n,k}$. 
		\end{lemma}
		
		\begin{proof}
			First, by Proposition~\ref{prop:RS_properies}.2, 
			$\Des(Q_\tau)=\Des(\tau)$, so $Q$ is descent-set-preserving. 
			
			Second, by 
			Proposition~\ref{prop:RS_properies3}, 
			$Q_\tau$ has $k$ odd columns
			for every $\tau\in\I_{n-k,0}\shuffle [n-k+1,\dots,n]$,
			so $Q$ maps $\I_{n-k,0}\shuffle [n-k+1,\dots,n]$ into $\SYT_{n,k}$. . 
			
			To prove that $Q$ is a bijection, we will construct an inverse.
			Assuming that $\tau \in \I_{n-k,0} \shuffle [n-k+1,\dots,n]$, 
			it will be shown that $\tau$ can be reconstructed from its recording tableau $Q_\tau$, which can be an arbitrary element of $\SYT_{n,k}$.
			Assume that $\tau \in \sigma \shuffle [n-k+1,\dots,n]$, where $\sigma \in \I_{n-k,0}$. 
			By Corollary~\ref{cor:RS_properties3b}, the largest $k$ letters in $P_\tau$ appear in the bottom cells of the $k$ odd columns of $P_\tau$, and they are increasing from left to right.
			These cells can be identified from $Q_\tau$, which has the same shape as $P_\tau$.
			We want to recover the positions of these $k$ largest letters in $\tau$, i.e., the values $\tau^{-1}(i)$ for $n-k+1 \le i \le n$. 
			Recall, from Proposition~\ref{prop:RS_properies}.1, that if $\tau$ corresponds (under RS) to the pair $(P_\tau,Q_\tau)$ then $\tau^{-1}$ corresponds to $(P_{\tau^{-1}},Q_{\tau^{-1}}) = (Q_\tau,P_\tau)$. 
			Apply to $Q_\tau$, $k$ times, the inverse of the RS insertion algorithm, as in~\cite[proof of Theorem 3.1.1]{Sagan_book}. Here is the first step:
			\begin{itemize}
				\item 
				Let $T_n := Q_\tau$. Assume that the bottom cell of the the rightmost odd column is in the $r^{th}$ row. 
				Let  $i_r$ be the entry in this cell.  
				\item 
				Let  $i_{r-1}$ be the largest letter smaller than $i_r$ in the $(r-1)^{st}$ row. Delete $i_r$ from the $r^{th}$ row and replace $i_{r-1}$ by $i_r$. 
				\item 
				Repeat this step until $i_1$, the largest letter smaller than $i_2$ in the first row, is replaced by $i_2$.
				\item 
				The letter $i_1$ is the position of $n$ in $\tau$, namely $\tau^{-1}(n)$. 
			\end{itemize}
			Apply the same procedure to the resulting tableau $T_{n-1}$ (with $\tau^{-1}(n)$ removed) to find the position of $n-1$, namely $\tau^{-1}(n-1)$, and so on. 
			
			After $k$ steps, the positions of the increasing subsequence of $\tau$ consisting of the largest $k$ letters have been determined. 
			At this stage we get $T_{n-k}$, which is the $P$ tableau corresponding to the sequence $\tau^{-1}(1), \ldots, \tau^{-1}(n-k)$. 
			This sequence has the same relative order of letters as the sequence $\sigma^{-1}(1), \ldots, \sigma^{-1}(n-k)$.
			Replacing the $n-k$ letters in $T_{n-k}$ by $1,\ldots,n-k$ with the same relative order therefore yields $P_{\sigma^{-1}} = Q_\sigma$.
			It is clear from the algorithm that all the columns of $T_{n-k}$, and therefore of $Q_\sigma$, have even lengths. By Corollary~\ref{cor:RS_involution} (for $k = 0$) there is a unique fixed-point-free involution $\sigma$ with this $Q_\sigma$ as a $Q$ tableau.
			Since $\sigma$ is an involution,
			Proposition~\ref{prop:RS_properies}.1 implies that $P_\sigma = Q_\sigma$, and therefore $\sigma$ is the RS preimage of $(Q_\sigma,Q_\sigma)$.
			This completes the proof.
		\end{proof}
		
		\begin{example}
			Let 
			\[
			Q_\tau = \young(1246,358,7) \,\, .
			\]
			This tableau has $n=8$ cells and $k=2$ columns of odd length. Thus $\tau \in \sigma \shuffle [7,8]$ 
			for some $\sigma \in \I_{6,0}$. 
			
			Start with $T_8 = Q_\tau$.
			The bottom cell of the rightmost odd column in $T_8$ appears in the first row. Thus $r=1$, the entry there is $i_r=i_1=6$, and therefore $\tau^{-1}(8) = 6$. 
			The resulting tableau after deleting this letter is 
			\[
			T_7 = \young(124,358,7) \,\,.
			\]
			Now the bottom cell of the rightmost odd column appears in row $r=3$. The entry there is $i_3=7$, and consequently $i_2=5$ and $i_1=4$. Thus $\tau^{-1}(7) = 4$. 
			This yields
			\[
			T_6 = \young(125,378) \,\,.
			\]
			Standartization (by mapping the letters in $T_6$ to $\{1,\ldots,6\}$ in a monotone increasing fashion) gives
			\[
			Q_\sigma = \young(124,356)\ .
			\]
			Hence $\sigma = {\rm RS}^{-1}(Q_\sigma,Q_\sigma) = [3,5,1,6,2,4]$ and $\tau =  [3,5,1,7,6,8,2,4]$. 
		\end{example} 
		
		\medskip 
		
		Recall the map $q$ from Definition~\ref{def:extended_iota}.4: 
		for $\tau\in \I_{n-k,0} \shuffle [n-k+1,\dots,n]$,  
		the involution $q(\tau)$ is the RS preimage of $(Q_\tau,Q_\tau)$, where $Q_\tau$ is the RS recording tableau of $\tau$. 
		
		\begin{corollary}\label{shuffles-RS}
			The map $q$ is a descent set and nesting number preserving 
			bijection from the set of shuffles $\I_{n-k,0} \shuffle [n-k+1,\ldots n]$ to 
			the set of involutions $\I_{n,k}$.
		\end{corollary}
		
		\begin{proof}
			Let $\tau \in \I_{n-k,0} \shuffle [n-k+1,\dots,n]$. 
			First, by 
			Proposition~\ref{prop:RS_properies3}, $Q_\tau$ has $k$ odd columns. 
			Combining this with Propositions~\ref{prop:RS_properies}.1 and~\ref{RS_involution_odd}, the RS preimage of $(Q_\tau,Q_\tau)$ is an involution with $k$ fixed points, namely, $q(\tau)\in \I_{n,k}$.  
			Moreover, by Lemma~\ref{lem:Q} together with Corollary~\ref{cor:RS_involution}, $q$ is a descent set preserving bijection. 
			Finally, by definition, $Q_\tau = Q_{q(\tau)}$. 
			By Lemma~\ref{lem:1} 
			for the involution
			$q(\tau)$, 
			\[
			\nest(q(\tau)) 
			= \lfloor \height(Q_{q(\tau)}) / 2 \rfloor.
			\]
			Thus, by Corollary~\ref{t.tau},
		\[
		\nest(\tau) 
		= \lfloor \height(Q_{\tau}) / 2 \rfloor
		= \lfloor \height(Q_{q(\tau)}) / 2 \rfloor
		= \nest(q(\tau)). 
		\]
	\end{proof}

	\begin{proof}[Proof of Proposition~\ref{prop:bijection}]
		By Lemma~\ref{lem:shuffles} and Corollary~\ref{shuffles-RS}, 
		\[
		\hat\iota: \I_{n,k} \overset{\varphi} \to \I_{n-k,0}\shuffle [n-k+1,\dots,n] \overset{q} \to \I_{n,k}  
		\]
		is a bijection which satisfies
		\[
		\MDes(\pi) = \Des(\varphi(\pi)) = \Des(q(\varphi(\pi))) = \Des(\hat\iota(\pi))
		\]
		as well as
		\[
		\crn (\pi)
		= \nest(\varphi(\pi)) 
		= \nest(q(\varphi(\pi)))
		= \nest(\hat\iota(\pi)).
		\]
	\end{proof}
	

	
	

	\section{Schur-positivity}\label{sec:main0}

	\subsection{Background}\label{sec:Schur_background}
	
	Schur functions indexed by partitions of $n$ form a distinguished
	basis for $\Lambda_n$, the vector space of homogeneous symmetric
	functions of degree $n$; see, e.g., \cite[Corollary 7.10.6]{EC2}. 
	A symmetric function in $\Lambda_n$ is {\em
		Schur-positive} if all the coefficients in its expansion in the
	basis $\{s_{\lambda} \,:\, \lambda \vdash n\}$ of Schur functions
	are nonnegative.  
	
	
	For each $D \subseteq [n-1]=\{1,2,\dots,n-1\}$, define the {\em fundamental quasisymmetric function}
	\[
	F_{n,D}({\bf x}) := \sum_{\substack{i_1\le i_2 \le \ldots \le i_n \\ i_j <
			i_{j+1} \text{ if } j \in D}} x_{i_1} x_{i_2} \cdots x_{i_n}.
	\]
	Let $A$ be a set of combinatorial objects, equipped with a set-valued function $D : A \rightarrow 2^{[n-1]}$.  
	We say that $A$ is {\em symmetric} ({\em Schur-positive}) with respect to $D$ if 
	\[
	\Q_{A,D}:=\sum\limits_{\pi\in A} F_{n, D(\pi)}
	\]
	is a symmetric (respectively, Schur-positive) function. 
	Determining whether a given symmetric  
	(quasisymmetric) function is Schur-positive is a major problem in contemporary algebraic combinatorics~\cite[\S 3]{Stanley_problems}.
	
	The following theorem is due to Gessel.
	
	\begin{theorem}\label{thm:Gessel1}\cite[Theorem 7.19.7]{EC2}
		For every partition $\lambda\vdash n$,
		\[
		\Q_{\SYT(\lambda), \Des}=s_\lambda.
		\]
		Thus $\SYT(\lambda)$ is symmetric and Schur-positive with respect to the standard descent set.
	\end{theorem}
	
	
	We say that a 
	statistic $f:A\rightarrow \NN\cup\{0\}$ 
	is Schur-positive 
	on $A$ with respect to the set-valued function $D$ if 
	\[
	\sum\limits_{\pi\in A} q^{f(\pi)} 
	F_{n,D(\pi)}, 
	\]
	is a Schur-positive symmetric function. 
	Examples of Schur-positive statistics with respect to the standard descent set on permutations include 	
	\begin{itemize}
		
		\item
		Statistics on $\symm_n$ which are invariant under conjugation; e.g., 
		the cycle number and the number of fixed points. This follows from~\cite[Theorem 2.1]{GR}. 
		
		\item
		Statistics on $\symm_n$ which are invariant under Knuth relations; e.g., the
		length of the longest increasing subsequence, the inverse descent
		number, and the inverse major index. This follows from Theorem~\ref{thm:Gessel1} above together with Proposition~\ref{prop:RS_properies}. 
		
		\item
		The inversion number on $symm_n$ (reduced to the inverse major index by Foata's
		bijection). For a far reaching generalization see~\cite[Theorem 6.3]{Sh-Wachs}. 
		
	\end{itemize}
	
	
	Theorem~\ref{thm:main0}, to be proved in the following subsection, implies that on the set $\M_n$ of all matchings on $n$ points, the pair $(\crn,\um)$ of the crossing number and the number of unmatched points is Schur-positive with respect to the geometric descent set $\MDes$.
	
	\subsection{Proof of Theorem~\ref{thm:main0}}\label{sec:proof_main0}
	
	
	Recall from Section~\ref{sec:introduction} the following notations:
	$\um(m)$ is the number of unmatched points in a matching $m \in \M_n$;  
	and, for a partition $\lambda$,  $\height(\lambda)$ is the number of parts in $\lambda$ and 
	$\oddcol(\lambda)$ is the number of odd parts in the conjugate partition.
	
	The following proposition follows from 
	Theorem~\ref{cor:main11}.
	
	\begin{proposition}\label{thm:main00}
		For every $n\ge 0$
		\begin{equation}
			\sum\limits_{m\in \M_{n}} q^{\um(m)}t^{\crn(m)} {\bf x}^{\MDes(m)}= \sum\limits_{\lambda\vdash n}  q^{\oddcol(\lambda)} t^{\lfloor \height(\lambda)/2  \rfloor} \sum\limits_{T\in \SYT(\lambda)}{\bf x}^{\Des(T)}. 
		\end{equation}
	\end{proposition}
	
	
	\begin{proof}
		We have 
		\[
		\begin{aligned}
			\sum\limits_{m\in \M_{n}} q^{\um(m)}t^{\crn(m)} {\bf x}^{\MDes(m)}&=
			\sum\limits_{k=0}^n q^k \sum\limits_{m\in \M_{n,k}} t^{\crn(m)} {\bf x}^{\MDes(m)}\\
			&=
			\sum\limits_{k=0}^n q^k \sum\limits_{m\in \M_{n,k}} t^{\nest(m)} {\bf x}^{\Des(m)}\\
			&= \sum\limits_{m\in \M_{n}} q^{\um(m)} t^{\nest(m)} {\bf x}^{\Des(m)}\\
			&=\sum\limits_{\lambda\vdash n}^n q^{\oddcol(\lambda)} t^{\lfloor \height(\lambda)/2 \rfloor} \sum\limits_{T\in \SYT(\lambda)}  {\bf x}^{\Des(T)}.
		\end{aligned}
		\]
		The second equality follows from Theorem~\ref{cor:main11}. 
		The last equality is obtained from the interpretation of matchings as involutions, followed by the bijection to SYT via the Robinson-Schensted correspondence, using
		Corollary~\ref{cor:RS_involution}
		and Lemma~\ref{lem:1}.
	\end{proof}


	\begin{proof}[Proof of Theorem~\ref{thm:main0}]
		Consider 
		the equality in Proposition~\ref{thm:main00}.
		Applying the vector space isomorphism from the ring of multilinear polynomials 
		to the ring of quasisymmetric functions, defined by ${\bf x}^J\mapsto F_{n,J}$ for every subset 
		$J\subseteq [n-1]$, one obtains  
		\[
		\begin{aligned}
			\sum\limits_{m\in \M_{n}} q^{\um(m)}t^{\crn(m)} F_{n,\MDes(m)} &= \sum\limits_{\lambda\vdash n}  q^{\oddcol(\lambda)} t^{\lfloor \height(\lambda)/2 \rfloor} \sum\limits_{T\in \SYT(\lambda)}F_{n,\Des(T)}\\
			&= \sum\limits_{\lambda\vdash n}  q^{\oddcol(\lambda)} t^{\lfloor \height(\lambda)/2 \rfloor} s_\lambda.
		\end{aligned}
		\]
		The last equality follows from Theorem~\ref{thm:Gessel1}.
	\end{proof}
	
	
	

	\section{Cyclic descent extensions}\label{sec:CDE}
	
	The above setting is applied in this section to 
	construct a cyclic descent extension for conjugacy classes of involutions and their refinements,  
	that is, involutions with fixed cycle structure and nesting number. 

		
		
		
		Let $\M_n$ be the set of matchings on $n$ points on the circle, labeled by $1,\dots, n$
		counterclockwise. Let $r:\M_{n}\rightarrow \M_{n}$
		be the counterclockwise rotation by $2\pi/n$. 
		Recall the definition of the geometric cyclic descent set map of a matching, $\cMDes: \M_n\mapsto 2^{[n]}$, from Definition~\ref{def:chord_des}.
		
		\begin{observation}\label{obs:CGDes}
			For every $m\in \M_n$
			\[
			\cMDes(m)\cap [n-1] = \MDes(m) 
			\]
			and
			\[
			\cMDes(r (m))=1+\cMDes(m),
			\]
			where addition is modulo $n$.
		\end{observation}
		
		In order to verify the non-Escher axiom for $\cMDes$, we need the following lemma.
		
		\begin{lemma}\label{lem:cMDes_emty_full}
			For $m \in \M_{n,k}$, where $n \ge k \ge 0$ with $n-k$ even, 
			\begin{enumerate}
				\item[(a)]
				$\cMDes(m) = \varnothing$ if and only if $k = n$, namely, $m$ contains no chords; and
				\item[(b)]
				$\cMDes(m) = [n]$ if and only if $k = 0$ and $\crn(m) = n/2$, namely, $n$ is even and $m$ matches $i$ with $i+n/2$ for any $1\le i\le n/2$.
			\end{enumerate}
		\end{lemma}
		
		\begin{proof}
			Consider the possible values of $k$.
			
			\begin{itemize}
				
				\item[{\bf Case 1.}] 
				If $k=n$ then all the points in $m$ are unmatched, and therefore $\cMDes(m) = \varnothing$.
				
				\item[{\bf Case 2.}] 
				If $0<k<n$ then $m$ has matched as well as unmatched points. There must be an umatched point followed by a matched one, and a matched point followed by an unmatched one. Therefore $\cMDes(m) \ne \varnothing, [n]$. 
				
				
				\item[{\bf Case 3.}] 
				If $k=0$ then all the points are matched and $n$ is even. 
				If $\cMDes(m)=[n]$ then, in particular, $\MDes(m)=[n-1]$ and, by Proposition~\ref{t:111}, $\Des(\iota(m))=[n-1]$. It follows that $\iota(m)=w_0=(1,n)(2,n-1)\cdots (n/2,n/2+1) \in \I_{n,0}$, hence $m$  matches $i$ with $i+n/2$ for any $1\le i\le n/2$ and $\crn(m)=n/2$. 
				For the opposite direction, if $\crn(m)=n/2$ then, by Definition~\ref{defn:cr-nest}, $m$  matches $i$ with $i+n/2$ for any $1\le i\le n/2$ 
				and $\cMDes(m)=[n]$.

			\end{itemize}
			
		\end{proof}

		Denote now
		\[
		\I_{n,k,j}:=\{\pi\in \I_{n,k},\ \nest(\pi)=j\}, 
		\]
		and recall the map $\hat\iota:\I_{n,k} \to \I_{n,k}$ from Definition~\ref{def:extended_iota}.
		
		\begin{proposition}\label{prop:CDE_involutions}
			Assume that $n \ge k \ge 0$ with $n-k$ even, and $0 \le j \le (n-k)/2$.
			\begin{itemize}
				\item[(a)] 
				If $0 < k < n$, 
				or $k=0$ and $j \ne n/2$,
				then the pair 
				\[
				(\cMDes\circ \hat\iota^{-1}, \hat\iota\circ r \circ \hat\iota^{-1})
				\] 
				is a (non-Escherian) cyclic extension of $\Des$ on $\I_{n,k,j}$.
				\item[(b)] 
				If $k=n$ (and necessarily $j=0$), 
				or $k=0$ and $j=n/2$, 
				then the above pair is an Escherian cyclic extension of $\Des$ on $\I_{n,k,j}$.
			\end{itemize}
		\end{proposition}
		
		\begin{proof}
			The number of unmatched points is invariant under rotation and (by Proposition~\ref{prop:bijection}) also under $\hat\iota$, 
			hence $\hat\iota\circ r \circ \hat\iota^{-1}(\pi)\in \I_{n,k}$ for every $\pi\in \I_{n,k}$. 
			Furthermore, 
			\[
			\pi\in \I_{n,k,j} \Longrightarrow
			\hat\iota\circ r \circ \hat\iota^{-1}(\pi)\in \I_{n,k,j} 
			\]
			since
			\[
			\nest (\hat\iota\circ r \circ \hat\iota^{-1}(\pi))= \crn (r \circ \hat\iota^{-1}(\pi))= \crn (\hat\iota^{-1}(\pi))=\nest(\hat\iota\circ\hat\iota^{-1}(\pi))=\nest(\pi).
			\] 
			Here we applied Proposition~\ref{prop:bijection} and the fact that the crossing number (but not the nesting number!) is invariant under rotation. 
			
			Denote
			\[
			\cDes(\pi):=\cMDes(\hat\iota^{-1}(\pi)) \qquad (\forall \pi\in \I_{n,k,j}).
			\]
			By Proposition~\ref{prop:bijection} and Observation~\ref{obs:CGDes} we have   
			\[
			\cDes(\pi)\cap [n-1]=\cMDes ( \hat\iota^{-1}(\pi))\cap[n-1]=
			\MDes (\hat\iota^{-1}(\pi))=\Des(\pi) 
			\]
			and
			\[
			\cDes(\hat\iota\circ r \circ \hat\iota^{-1}(\pi)) =
			\cMDes(r \circ \hat\iota^{-1}(\pi))
			= 1+\cMDes ( \hat\iota^{-1}(\pi))
			= 1+\cDes (\pi)
			\]
			for any $\pi\in \I_{n,k,j}$. 
			This proves the extension and equivariance properties for every $0\le k\le n$ and $0\ \le j \le (n-k)/2$. 
			Finally, by Lemma~\ref{lem:cMDes_emty_full}, the non-Escher property holds if and only if either $0<k<n$ or $k=0$ and $j\ne n/2$.
		\end{proof}

		\noindent
		{\it Proof of Proposition~\ref{cor:main}.}  Follows from Proposition~\ref{prop:CDE_involutions}.
		\qed
		
		\medskip
		
		
		Recall the map $Q: \symm_n \to \SYT_n$ sending each $\pi\in \symm_n$ to the corresponding RS recording tableau $Q_\pi$,
		and define $h:\I_{n,k}\mapsto \SYT_{n,k}$ by $h:=Q\circ \hat\iota$.
		A cyclic descent extension on the set 
		\[
		\SYT_{n,k,j}:=\{T\in \SYT_{n,k},\ 2j\le \height(T)\le 2j+1\}.
		\]
		is described in the following statement. 
		


		\begin{proposition}\label{prop:CDE_SYT}
			Assume that $n \ge k \ge 0$ with $n-k$ even, and $0 \le j \le (n-k)/2$.
			\begin{itemize}
				\item[(a)] 
				If $0 < k < n$, 
				or $k=0$ and $j \ne n/2$,
				then the pair 
				\[
				(\cMDes\circ h^{-1}, h \circ r \circ h^{-1})
				\] 
				is a (non-Escherian) cyclic extension of $\Des$ on $\SYT_{n,k,j}$.
				\item[(b)] 
				If $k=n$ (and necessarily $j=0$), 
				or $k=0$ and $j=n/2$, 
				then the above pair is an Escherian cyclic extension of $\Des$ on $\SYT_{n,k,j}$.
			\end{itemize}
		\end{proposition}
		
		\begin{proof}
			Follows from Proposition~\ref{prop:CDE_involutions}, noting that, 
			by Proposition~\ref{prop:RS_properies} and Lemma~\ref{lem:1}, the restriction of $Q$ to $\I_{n,k,j}$ is a descent set preserving bijection onto $\SYT_{n,k,j}$. 
		\end{proof}
		
		
		\begin{remark}
			Cyclic rotation of geometric configurations has been used before 
			for the construction of cyclic descent extensions 
			on standard Young tableaux of certain given shapes ---
			rectangular shapes of height at most $3$~\cite{Petersen_PR} 
			and flag shapes~\cite{Pechenik}.  
			These results motivated our work, and some of them are indeed obtained as special cases: 
			
			\begin{itemize}
				
				\item 
				Letting $k=0$ and $j=1$ in Proposition~\ref{prop:CDE_SYT}  
				yields a cyclic descent extension on standard Young tableaux of shape $(n,n)$, 
				since $\SYT_{2n,0,1}=\SYT(n,n)$. 
				One can verify that this cyclic extension 
				coincides with the one determined by Dennis White~\cite[Theorem 1]{Petersen_PR}.
				
				\item  
				Recalling that the number of Motzkin paths of length $n$ is equal to the number of standard Young tableaux of size $n$ and at most three rows~\cite{Reg81, Eu10, AR15},  
				consider Proposition~\ref{prop:CDE_SYT} on the union of $\cup_k \SYT_{n,k,1}$, namely $j=1$ and $k$ arbitrary. 
				This determines a cyclic descent extension on Motzkin paths via Han's bijection~\cite{Han}, which coincides with Han's cyclic descent extension on Motzkin paths.
			\end{itemize} 
			
		\end{remark}

		
		
		

		\section{Equidistribution revisited}\label{sec:Gessel}
		
		

		
		In an early version of this paper, the following conjecture was posed.
		
		
		\begin{conjecture}\label{conj:Gessel}
			Let  $\mu\vdash m$ and $\nu\vdash n$  be integer partitions with no common part. 
			Let $\pi$ and $\sigma$ be permutations of cycle types $\mu$ and $\nu$,   respectively, with disjoint supports.  
			Let $A_{\pi,\sigma}$ be the subset of the conjugacy class of cycle type $\mu\sqcup \nu\vdash m+n$ consisting of the permutations for which the relative order of the letters in the union of all cycles of $\mu$ is as in $\pi$, and the relative order of the letters in the union of all cycles of $\nu$ is as in $\sigma$. 
			Then
			\[
			\sum\limits_{w\in A_{\pi,\sigma}} {\bf x}^{\Des(w)}=\sum\limits_{\tau\in \pi\shuffle \sigma} {\bf x}^{\Des(\tau)}.
			\]
		\end{conjecture}
		
		\begin{example}
			Let $\pi=(1,3,2)$ and $\sigma=(4)$. 
			Then $A_{\pi,\sigma}$ is the following subset of the conjugacy class of cycle type $(3,1)$ in $S_4$:
			\[
			A_{\pi,\sigma}=\{(1,3,2)(4),\ (1,4,2)(3),\ (1,4,3)(2),\ (2,4,3)(1) \}=\{[3124],  \ [4132],\ [4213],\ [1423] \}
			\]
			This set of permutations and the set
			\[
			\pi\shuffle \sigma= [312]\shuffle [4]=\{[3124],\ [3142],\ [3412],\ [4312]\} 
			\]
			have the same distribution of the descent set.
		\end{example}
		
		\smallskip
		
		Conjecture~\ref{conj:Gessel} was proved by Gessel. 
		\begin{proposition}\label{lem:Gessel}\cite{Gessel}
			Conjecture~\ref{conj:Gessel} holds. 
		\end{proposition}
		
		The proof is partly algebraic and not bijective. 
		
		
		\smallskip
		
		Proposition~\ref{lem:Gessel} implies the following. 
		
		\begin{corollary}\label{cor:shuffles_implicit}
			There exists an (implicit) descent set, nesting number and crossing number preserving bijection 
			\[
			\phi: \I_{n-k,0}\shuffle [n-k+1,\dots,n] \to \I_{n,k}.
			\]
		\end{corollary}

		\begin{proof} 
			Take, in Proposition~\ref{lem:Gessel}, 
			$\pi = id_k \in \symm_k$ and $\sigma\in \I_{n-k,0}$. 
			It follows that there exists a bijection 
			\[
			\phi: \I_{n-k,0}\shuffle [n-k+1,\dots,n] \to \I_{n,k}
			\]
			which preserves the descent set and satisfies the following property: 
			for every $\sigma\in \I_{n-k,0}$ and every $\tau\in \sigma \shuffle [n-k+1,\dots,n]$, the relative order of the letters in the union of all  $2$-cycles in $\phi(\tau)$ is equal to the relative order of the letters in $\sigma$.  
			By Definition~\ref{def:nest_shuffles}, the nesting and crossing numbers of $\tau$ are the same as those of $\sigma$. Thus $\phi$ preserves nesting and crossing numbers as well. 
		\end{proof}
		
		\begin{remark}
			The explicit bijection $q:\I_{n-k,0}\shuffle [n-k+1,n-k+2,\dots n]\rightarrow \I_{n,k}$ from Lemma~\ref{shuffles-RS} preserves the descent set and the nesting number, but does not preserve the crossing number. 
			The bijection $\phi:\I_{n-k,0}\shuffle [n-k+1,n-k+2,\dots n]\rightarrow \I_{n,k}$, whose existence is claimed in Corollary~\ref{cor:shuffles_implicit}, preserves the crossing number as well. 
		\end{remark}
		
		The following 
		refinement of Theorem~\ref{cor:main11} follows.
		
		\begin{theorem}\label{cor:main111}
			For every $n\ge k\ge 0$ with $n-k$ even,
			\[
			\sum_{m \in \M_{n,k}} q^{\crn(m)} t^{\nest(m)}
			{\bf x}^{\MDes(m)} 
			= \sum_{m \in \M_{n,k}} q^{\nest(m)} t^{\crn(m)}
			{\bf x}^{\Des(m)}. 
			\]
		\end{theorem}
		
		\begin{proof} 
			Replace $\hat\iota:=q\circ \varphi$ by $\eta:= \phi\circ \varphi$ in the proof of Theorem~\ref{cor:main11}, where   
			again we use $\I_{n,k}$ instead of $\M_{n,k}$.  Combining Lemma~\ref{lem:shuffles} 
			with 
			Corollary~\ref{cor:shuffles_implicit} implies that 
			for any $n \ge k \ge 0$,  
			the map 
			\[
			\eta: \I_{n,k} \rightarrow \I_{n,k}
			\]
			is a bijection which satisfies 
			\[
			\MDes(\pi) = \Des(\eta(\pi)) 
			\qquad (\forall \pi\in \I_{n,k})
			\]
			as well as
			\[
			\nest(\pi)=\crn(\eta(\pi)) 
			\ \ \ {\rm{and}} \quad 
			\crn(\pi)=\nest(\eta(\pi)) 
			\qquad (\forall \pi\in \I_{n,k}). 
			\]
			This completes the proof. 
		\end{proof}
		
		\begin{problem}
			Find an explicit bijective proof of Theorem~\ref{cor:main111}. 
		\end{problem}
		
		\bigskip
		
		\noindent
		{\bf Acknowledgements.} 
		The authors thank  
		Ira Gessel, Bin Han and Tom Roby for useful discussions and contributions and   
		Martin Rubey, Bruce Sagan, Richard Stanley and Sheila Sundaram 
		for helpful references, suggestions and comments

		
		
		

	\end{document}